\documentclass[leqno, 11pt]{amsart}
\usepackage[margin=1.3in]{geometry}

\usepackage[utf8]{inputenc}
\usepackage{amsmath}
\usepackage[colorlinks]{hyperref}
\usepackage{amstext}
\usepackage{amscd}
\usepackage{amsmath}
\usepackage{latexsym}
\usepackage{url}
\usepackage{amsthm}
\usepackage{amssymb}
\usepackage[all]{xy}
\usepackage{mathrsfs}
\usepackage{color}
\usepackage{tikz}
\usepackage{indentfirst}

\newtheorem{theorem}{\indent\sc Theorem}[section]
\newtheorem{thmx}{Theorem}
\newtheorem{lemma}[theorem]{\indent\sc Lemma}

\newtheorem{proposition}[theorem]{\indent\sc Proposition}

\theoremstyle{definition} 
\newtheorem{definition}[theorem]{\indent\sc Definition}
\newtheorem{remark}[theorem]{\indent\sc Remark}

\numberwithin{equation}{section}

\hypersetup{
     colorlinks,breaklinks,
     citecolor    = [RGB]{0,50,175},
     linkcolor =[RGB]{150,10,0}
}

\begin{document}
\title[{Relative RCC threefold char. $p$}]{On relative rational chain connectedness of threefolds with anti-big canonical divisors in positive characteristics}
\author{Yuan Wang}
%
\address{Department of Mathematics  \endgraf
University of Utah  \endgraf 
155 South 1400 East  \endgraf 
Salt Lake City, UT 84112-0090  \endgraf
USA}
\email{ywang@math.utah.edu}
\thanks{The author was supported in part by the FRG grant DMS-\#1265261.}
\begin{abstract}
In this paper we prove two results about the rational chain connectedness for klt threefolds with anti-big canonical divisors in the relative setting.
\end{abstract}
\maketitle
\section{Introduction}
It is widely recognized that the geometry of a higher-dimensional variety is closely related to the geometry of rational curves on it. A classical result in the early 90s by Campana (\cite{Campana92}) and Koll\'{a}r-Miyaoka-Mori (\cite{KMM92}) says that smooth Fano varieties are rationally connected in characteristic zero and are rationally chain connected in positive characteristics. This was later generalized by Zhang (\cite{Zhang06}) and Hacon-McKernan (\cite{HM07}) in characteristic zero. More recently using the minimal model program by Hacon-Xu (\cite{HX15}) and Birkar (\cite{Birkar13}), Gongyo-Li-Patakfalvi-Schwede-Tanaka-Zong (\cite{GLPSTZ15}) proved that projective globally $F$-regular threefolds in characteristic $\ge 11$ are rationally chain connected and this was later generalized to log Fano type threefold by Gongyo-Nakamura-Tanaka (\cite{GNT15}). 

The main result in \cite{HM07} is as follows.
\begin{theorem}\cite[Theorem 1.2]{HM07}\label{Shokurov RCC}
Let $(X,\Delta)$ be a log pair, and let $f:X\to S$ be a proper morphism such that $-K_X$ is relatively big and $-(K_X+\Delta)$ is relatively semiample. Let $g:Y\to X$ be any birational morphism. Then the connected components of every fiber of $f\circ g$ are rationally chain connected modulo the inverse image of the locus of log canonical singularities of $(X,\Delta)$.
\end{theorem}

In this paper we prove a theorem similar to Theorem \ref{Shokurov RCC} for morphisms from a klt threefold to a variety of dimension $\ge 1$. More precisely, we have
\begin{thmx}[Theorem \ref{relativeRCC}] \label{1}
Let $X$ be a normal $\mathbb{Q}$-factorial threefold over an algebraically closed field $k$ of characteristic $\ge 7$ and $(X,D)$ a klt pair. Let $f: X\to Z$ be a proper morphism such that $f_*\mathcal{O}_X=\mathcal{O}_Z$, ${\rm dim}(Z)=1$ or $2$, $Z$ is klt, $-K_X$ is relatively big, $-(K_X+D)$ is relatively semi-ample and $(X_z,D_z)$ is klt for general $z\in Z$. Let $g:Y\to X$ be any birational morphism. Then the connected components of every fiber of $f\circ g$ are rationally chain connected.
\end{thmx}
Motivated by Theorem \ref{1}, we construct a global version of rational chain connectedness for threefolds. 
\begin{thmx}[Theorem \ref{globalRCC}]\label{globalRCCB}
Let $X$ be a projective threefold over an algebraically closed field $k$ of characteristic $p>0$, $f:X\to Y$ a projective surjective morphism from $X$ to a projective variety $Y$ such that $f_*\mathcal{O}_X=\mathcal{O}_Y$. Let $D$ be an effective $\mathbb{Q}$-divisor, and $X_{\overline{\eta}}$ the geometric generic fiber of $f$. Assume that the following conditions hold. 
\begin{enumerate}
\item $(X,D)$ is klt, $-K_X$ is big and $f$-ample, $K_X+D\sim_{\mathbb{Q}}0$ and the general fibers of $f$ are smooth.
\item $\displaystyle p> \frac{2}{\delta}$, where $\delta$ is the minmum non-zero coefficient of $D$.
\item $D=E+f^*L$ where $E$ is an effective $\mathbb{Q}$-Cartier divisor such that $p\nmid {\rm ind}(E)$, $(X_{\overline{\eta}},E|_{X_{\overline{\eta}}})$ is globally $F$-split, and $L$ is a big $\mathbb{Q}$-divisor on $Y$. \label{Fsplitcondition}
\item ${\rm dim}(Y)=1$ or $2$.
\end{enumerate}
Then $X$ is rationally chain connected.
\end{thmx}
Here ${\rm ind}(E)$ means the Cartier index of $E$. 

The main ingredients of the proofs of Theorem \ref{1} and Theorem \ref{globalRCCB} are the minimal model program constructed in \cite{HX15}, \cite{Birkar13} and \cite{GLPSTZ15}; some facts, especially Theorem 2.1, in \cite{GLPSTZ15}; some positivity results by Patakfalvi (\cite{Patakfalvi14}) and Ejiri (\cite{Ejiri15}); a canonical bundle formula constructed in Section \ref{CBFsection} in the spirit of the paper \cite{PS09} by Prokhorov and Shokurov. Note that the condition (\ref{Fsplitcondition}) in Theorem \ref{globalRCCB} is used in order to apply the result \cite[Theorem 1.1]{Ejiri15} of Ejiri to deduce that $-K_Y$ is big, and to apply Theorem \ref{CanBund} when $\dim Y=2$. This creates enough rational curves on $Y$. Note that by \cite[Example 3.4]{Ejiri15}, $(X_{\overline{\eta}},E|_{X_{\overline{\eta}}})$ being globally $F$-split is equivalent to $S^0(X_{\overline{\eta}},E|_{X_{\overline{\eta}}},\mathcal{O}_{X_{\overline{\eta}}})=H^0(X_{\overline{\eta}},\mathcal{O}_{X_{\overline{\eta}}})$.

Note that although the proof is independent, Theorem \ref{1} can be implied by Theorem 4.1 of the paper \cite{GNT15}, which was put on arXiv before this paper. 
The proof of \cite[Theorem 4.1]{GNT15} relies on the minimal model program in dimension $3$ in positive characteristic, which is only established in characteristic $\ge 7$ so far. On the other hand, Theorem \ref{globalRCCB} covers some cases in characteristic $<7$. In particular it does not rely on the minimal model program and is not implied by \cite[Theorem 4.1]{GNT15}.
\subsection*{Acknowledgements}
The author would like to express his gratitude to Christopher Hacon for suggesting this direction of research and a lot of valuable suggestions, comments, support and encouragement. He would like to thank Karl Schwede for answering many questions about $F$-singularities. He also thanks Omprokash Das, Honglu Fan and Zsolt Patakfalvi for helpful discussions. Finally he would like to thank the referee for many useful suggestions.
\section{Preliminaries}
We work over an algebraically closed field $k$ of characteristic $p>0$.
\subsection{Preliminaries on rational connected varieties and the minimal model program} 
\begin{definition}
For a variety $X$ and a $\mathbb{Q}$-Weil divisor on $X$ such that $K_X+\Delta$ is $\mathbb{Q}$-Cartier. Let $f:Y\to X$ be a log resolution of $(X,\Delta)$ and we write 
$$K_Y=f^*(K_X+\Delta)+\sum_ia_iE_i$$
where $E_i$ is a prime divisor. We say that $(X,\Delta)$ is 
\begin{itemize}
\item \emph{sub Kawamata log terminal} (\emph{sub-klt} for short) if $a_i>-1$ for any $i$.
\item \emph{Kawamata log terminal} (\emph{klt} for short) if $a_i>-1$ for any $i$ and $\Delta\ge 0$.
\item \emph{log canonical} if $a_i\ge -1$ for any $i$ and $\Delta\ge 0$.
\end{itemize}
\end{definition}
\begin{definition}\cite[IV.3.2]{Kollar96}
Suppose that $X$ is a variety over $k$.
\begin{enumerate}
\item We say that $X$ is \emph{rationally chain connected (RCC)} if there is a family of proper and connected algebraic curves $g:U\to Y$ whose geometric fibers have only rational components and there is a cycle morphism $u:U\to X$ such that $u^{(2)}:U\times_YU\to X\times_kX$ is dominant.
\item We say that $X$ is \emph{rationally connected (RC)} if (1) holds and moreover the geometric fibers of $g$ in (1) are irreducible.
\end{enumerate}
\begin{proposition} \label{MMP}
Let $X$ be a klt $\mathbb{Q}$-factorial threefold over an algebraically closed field $k$ and ${\rm char}(k)\ge 7$. Let $g: W\to X$ be a log resolution and assume that $K_W+E=g^*K_X+B$, where $E$ and $B$ are exceptional divisors and the coefficients in $E$ are all $1$. Then relative minimal model for $(W,E)$ over $X$ exists.  Denote this process by 
$$W=W_0 \overset{\text{$f_0$}}\dashrightarrow W_1\overset{\text{$f_1$}}\dashrightarrow ...\overset{\text{$f_{N-1}$}}\dashrightarrow W_N=W'.$$
Then we actually have $W'=X$. Moreover if we have a morphism $h:X\to Y$ such that every fiber of $h$ is RCC, then every fiber of $h\circ g$ is RCC.
\end{proposition}
\begin{proof}
The existence of this minimal model program is by \cite[Theorem 3.2]{GLPSTZ15}. So we have a morphism $g':W'\to X$ and we want to show that $g'$ is the identity. Denote the strict transform of $E$ by $E'$, then $K_{W'}+E'=g'^*K_X+B'$ for some exceptional $\mathbb{Q}$-divisor $B'$. By construction of the minimal model program we know that $g'^*K_X+B'$ is nef over $X$ which means that $B'$ is $g'$-nef and since $X$ is klt the support of $B'$ is the whole exceptional locus of $g'$. So we can get that $B'=0$ by negativity lemma, and since $X$ is $\mathbb{Q}$-factorial we will get $W'=X$.

The proof of the last statement follows the proof of \cite[Proposition 3.6]{GLPSTZ15}. Without loss of generality we can do a base change and assume that the base field $k$ is uncountable. Define $F$ in the following way: if $f_i$ is a divisorial contraction, then let $E_0=E$, $E_{i+1}=f_{i,*}E_i$ and $F$ an arbitrary component of $E_i$; if $f_i$ is a flip and $C$ is any flipping curve then let $F$ be a component of $E_i$ that contains $C$. Let $K_F+\Delta_F:=(K_{W_i}+E_i-\frac{1}{n}(E_i-F))|_F$ where $n\gg 0$. By assumption $K_{W_i}+E_i-\frac{1}{n}(E_i-F)$ is plt, then by adjunction $K_F+\Delta_F$ is klt, hence by \cite[Theorem 14.4]{Tanaka12} $F$ is $\mathbb{Q}$-factorial. We also know that $-(K_{W_i}+E_i)$ is $f_i$-ample by assumption, then $-(K_F+\Delta_F)$ is ample. Moreover by \cite[Corollary 2.2.8]{Prokhorov99} the coefficients of $\Delta_F$ are in the standard set $\{1-\frac{1}{n}|n\in \mathbb{N}\}$. Let $\tilde{F}$ be the normalization of $F$. Then by \cite[Theorem 3.1]{HX15} we know that $(\tilde{F},\Delta_{\tilde{F}})$ is strongly $F$-regular and by \cite[Theorem 4.1]{HX15} $F$ is a normal surface. 

Next we consider three cases. \\
\emph{Case 1:} If $f_i$ is a divisorial contraction and the exceptional divisor is contracted to a point, then since $-(K_F+\Delta_F)$ is ample, by \cite[Lemma 2.2]{Kawamata94} $F$ is a rational surface, in particular it is rationally connected. \\
\emph{Case 2:} If $f_i$ is a divisorial contraction and the exceptional divisor is contracted to a curve, then let $p:F\to B$ be the Stein factorization of $f_i|_F$. By assumption $-(K_F+\Delta_F)$ is $f_i$-ample, so it is $p$-ample. Then for a general fiber $D$ of $p$ we have
$$(K_F+D)\cdot D=(K_F+\Delta_F+D-\Delta_F)\cdot D=(K_F+\Delta_F)\cdot D-\Delta_F\cdot D<0.$$ 
 Here $D$ is reduced and irreducible by \cite[Theorem 7.1]{Badescu01}. Hence by \cite[Theorem 5.3]{Tanaka12} $D\cong\mathbb{P}^1$. Therefore every component of every fiber of $f_i$ is a rational curve. \\
\emph{Case 3:} If $f_i$ is a flip, then let $C$ be an arbitrary flipping curve. By assumption we have $(K_F+\Delta_F)\cdot C<0$, $C^2<0$ and $0\le {\rm coeff}_C\Delta_F<1$, so $(K_F+C)\cdot C<0$. Again by \cite[Theorem 5.3]{Tanaka12} $C\cong\mathbb{P}^1$. 

We denote a fiber of $h$ over $y\in Y$ by $F_{X,y}$. There is a morphism from $W_i$ to $Y$ for every $i$, and we denote denote the fiber of this morphism over $y$ as $F_{W_i, y}$. Then there is a rational map $F_{W_i, y}\dasharrow F_{W_{i+1}, y}$. From the above \emph{Case 1-3} we see that compared to $F_{W_i, y}$, there are only rational curves or a rational surface generated in $F_{W_{i+1}, y}$. So the RCC-ness of $F_{W_{i+1}, y}$ implies the RCC-ness of $F_{W_i, y}$. By assumption $F_{X,y}$ is RCC, so $F_{W,y}$ is RCC.
\end{proof}
\begin{proposition}\label{KXMMP}
Let $X$ be a klt $\mathbb{Q}$-factorial threefold over an algebraically closed field $k$ and ${\rm char}(k)\ge 7$. Let $f:X\to Y$ be a morphism from $X$ to a normal surface $Y$. Suppose we run a $K_X$-minimal model program and it terminates at $g:X'\to Y$. If every fiber of $g$ is RCC then every fiber of $f$ is RCC.
\end{proposition}
\begin{proof}
This can be easily deduced from Proposition \ref{MMP} by taking a common resolution of $X$ and $X'$. The proof of \cite[Proposition 3.6]{GLPSTZ15} works as well.
\end{proof}
\end{definition}
\subsection{Preliminaries on $F$-singularities}
In this article, for a proper variety $X$, a $\mathbb{Q}$-divisor $\Delta$ and the line bundle $M$ we will use the concepts of \emph{strongly $F$-regular}, the \emph{non $F$-pure ideal} $\sigma(X,\Delta)$ and $S^0(X,\sigma(X,\Delta)\otimes M)$. The definitions of these can be found in many papers related to $F$-singularities (e.g. \cite{HX15}). For a pair $(X,\Delta)$ where $\Delta$ is a $\mathbb{Q}$-Cartier divisor we also follow the definition of \emph{globally $F$-split} in \cite{Ejiri15}.
\begin{lemma}\label{basecurverational}
Let $X$ be a surface, $D$ an effective $\mathbb{Q}$-divisor on $X$, $f:X\to C$ a morphism from $X$ to a smooth curve $C$, and $(X_c,D_c)$ is a strongly $F$-regular pair for general $c\in C$. Assume that $-K_X$ is big, $K_X+D\sim_{\mathbb{Q}}0$, then $C\cong\mathbb{P}^1$.
\end{lemma}
\begin{proof}
By Kodaira's Lemma we can write $D\sim_{\mathbb{Q}}\epsilon f^*H+E$ where $H$ is an ample $\mathbb{Q}$-divisor on $C$, $0<\epsilon\in \mathbb{Q}$, $E$ is an effective $\mathbb{Q}$-divisor on $X$ and $(X_c,E_c)$ is also strongly $F$-regular for general $c\in C$ (since $X_c$ is a curve). Suppose that $C$ is not isomorphic to $\mathbb{P}^1$. We know that $K_{X/C}+E\sim_{\mathbb{Q}}f^*(-K_C-\epsilon H)$ is $f$-nef and $K_{X_c}+E_c$ is semi-ample for general $c\in C$, so by \cite[Theorem 3.16]{Patakfalvi14}, $K_{X/C}+E=K_X-f^*K_C+E$ is nef. Since we have assumed that $g(C)>0$ we have that $K_X+E$ is nef. However this is impossible since $K_X+E\sim_{\mathbb{Q}}-\epsilon f^*H$ where $H$ is ample and $\epsilon>0$.
\end{proof}
\subsection{Weak positivity}
Let $Y$ be a non-singular projective variety, $\mathcal{F}$ a torsion-free coherent sheaf on $Y$. We take $i:\hat{Y}\to Y$ to be the biggest open subvariety such that $\mathcal{F}|_{\hat{Y}}$ is locally free. Let $\hat{S}^k(\mathcal{F}):=i_*S^k(i^*\mathcal{F})$.
\begin{definition}\cite[Definition 1.2]{Viehweg83} \label{WP}
We call $\mathcal{F}$ \emph{weakly positive}, if there is an open subset $U\subseteq Y$ such that for every ample line bundle $\mathcal{H}$ on $Y$ and every positive number $\alpha$ there exists some positive number $\beta$ such that $\hat{S}^{\alpha\cdot\beta}(\mathcal{F})\otimes\mathcal{H}^{\beta}$ is generated by global sections over $U$.
\end{definition}
\begin{lemma}\label{WPimpliesnef}
Weakly positive line bundles are nef.
\end{lemma}
\begin{proof}
This easily follows from Definition \ref{WP}.
\end{proof}
\section{Relative rational chain connectedness}
In this section we prove the following
\begin{theorem} \label{relativeRCC}
Let $X$ be a normal $\mathbb{Q}$-factorial threefold over an algebraically closed field $k$ of characteristic $\ge 7$ and $(X,D)$ a klt pair. Let $f: X\to Z$ be a proper morphism such that $f_*\mathcal{O}_X=\mathcal{O}_Z$, ${\rm dim}(Z)=1$ or $2$, $Z$ is klt, $-K_X$ is relatively big, $-(K_X+D)$ is relatively semi-ample and $(X_z,D_z)$ is klt for general $z\in Z$. Let $g:Y\to X$ be any birational morphism. Then the connected components of every fiber of $f\circ g$ are rationally chain connected.
\end{theorem}
\begin{proof}
First we observe that $(X_z,D_z)$ being klt implies that $X_z$ is normal (in particular reduced) and irreducible.

Next we prove that if every fiber of $f$ is RCC, then every fiber of $f\circ g$ is RCC. We take a log resolution of $Y$ and denote it by $p: Y'\to Y$ and let $q=g\circ p$. If we have $K_{Y'}=q^*K_X+\tilde{B}$ then $K_{Y'}-\tilde{B}=q^*K_X$ and the coefficients of $-\tilde{B}$ are $<1$. Then we can add another effective divisor to make all the coefficients 1, and we denote this divisor by $\tilde{E}$. Now we run a relative $(K_{Y'}+\tilde{E})-$MMP of $Y'$ over $X$. By Proposition \ref{MMP} we see that if every fiber of $f$ is RCC then every fiber of $f\circ g\circ p$ is RCC, hence every fiber of $f\circ g$ is RCC.

Therefore it suffices to show that every fiber of $f$ is RCC. We consider the cases of ${\rm dim}(Z)=2$ and ${\rm dim}(Z)=1$ respectively. \\

\noindent \emph{Case 1: ${\rm dim}(Z)=2$.}

If ${\rm dim}(Z)=2$ then a general fiber of $f$ being normal and $-K_X$ being relatively big implies that a general fiber of $f$ is a smooth rational curve. Next we run a relative minimal model program over $Z$ and denote this process as
$$X=X_0 \overset{\text{$f_0$}}\dashrightarrow X_1\overset{\text{$f_1$}}\dashrightarrow ...\overset{\text{$f_{N-1}$}}\dashrightarrow X_n=X'.$$
Since $-K_X$ is relatively big we end up with a Mori fiber space $X'\xrightarrow{h} Z'\xrightarrow{p} Z$ where $Z'$ is also a surface. Then the general fibers of $h$ are rational curves. Moreover since $p_*\mathcal{O}_{Z'}=\mathcal{O}_Z$ we know that $p$ is birational.

Now we prove that $h$ is equidimensional. Suppose that this is not the case, then there is a fiber $\tilde{F}$ of $h$ over a point $\tilde{z}\in Z'$ which contains a $2$-dimensional irreducible component. If $\tilde{F}$ is reducible then let $\tilde{F}_1$ be a $2$-dimensional component of $\tilde{F}$ and $\tilde{F}_2$ another component which intersects $\tilde{F}_1$. We can choose a curve $\tilde{C_2}\subseteq \tilde{F}_2$ such that $\tilde{F}_1\cdot\tilde{C_2}>0$. On the other hand if we take a general point $z'\in Z'$ then $h^{-1}(z')$ is an irreducible curve and $h^{-1}(z')\cdot\tilde{F}_2=0$. This is a contradiction to the fact that $\rho(X'/Z')=1$. If $\tilde{F}$ is irreducible, by Bertini's Theorem we have a very ample divisor $H\subset X'$ such that $H\cap \tilde{F}$ is an irreducible curve which we denote by $\tilde{C}$. We do the Stein factorization of $h|_H$ and denote the process as 
$$H\xrightarrow{h_1}Z''\xrightarrow{h_2}Z',$$
then $h_1$ is birational and $\tilde{C}$ is an exceptional curve of $h_1$. After possibly replacing $Z''$ by its normalization we can assume that $Z''$ is normal. Now $\tilde{F}\cdot\tilde{C}$ is equal to $\tilde{C}^2$ viewed as the self-intersection of $\tilde{C}$ in $H$, so by the Negativity Lemma it is $<0$. On the other hand we can still take a general point $z'\in Z'$ as above such that $h^{-1}(z')\cdot\tilde{F}=0$. This is also a contradiction to the fact that $\rho(X'/Z')=1$.

Since $h$ is equidimensional, by \cite[Lemma 3.7]{Debarre01} the components of every fiber of $h$ are rational curves. Then by Proposition \ref{KXMMP} every fiber of $f$ is RCC. \\

\noindent \emph{Case 2: ${\rm dim}(Z)=1$.}

Without loss of generality we can do a base change and assume that the base field $k$ is uncountable. By passing to the normalization of $Z$ we can assume that $Z$ is smooth. Then since every closed point of $Z$ is a Cartier divisor, every fiber of $f$ is also Cartier, hence $f$ is equidimensional. 

We first show that the general fibers of $f$ are rationally chain connected. Let $F$ be a general fiber of $f$. Since we assume that $(F, D|_F)$ is klt, by adjunction we know that
$$K_X|_F\equiv_{\rm num}(K_X+F)|_F=K_F+{\rm Diff}_F(0),$$
where ${\rm Diff}_F(0)\ge 0$ (cf. \cite[Proposition-Definition 16.5]{Kollar92}). So $-(K_F+{\rm Diff}_F(0))$, hence $-K_F$, is big. Therefore $\kappa(F)=-\infty$ and $F$ is birationally ruled by classification of surfaces. To prove that the general fibers of $f$ are RCC it suffices to prove that $F$ is rational. By assumption $-(K_F+D|_F)=-(K_X+D)|_F$ is semiample, so there exists an effective $\mathbb{Q}$-divisor $H$ such that $H\sim_{\mathbb{Q}}-(K_F+D|_F)$ and $(F,D|_F+H)$ is klt. We define $\Delta:=D|_F+H$. Let $\pi:F'\to F$ be a minimal resolution of $(F, {\rm Diff}_F(0))$, then $F'$ maps to a ruled surface $F''$ over a smooth curve $B$ via a sequence of blow-downs and we denote the morphism by $\psi$. The situation is as follows.
\begin{center}
\begin{tikzpicture}[scale=1.6]
\node (A) at (0,0) {$F$};
\node (B) at (1,1) {$F'$};
\node (C) at (2,0) {$F''$};
\node (D) at (2,-1) {$B$};  
\path[->,font=\scriptsize]
(B) edge node[above]{$\pi$} (A)
(B) edge node[above]{$\psi$} (C)
(C) edge node[right]{$q$} (D);
\end{tikzpicture} 
\end{center}    
Since we have that $(F,\Delta)$ is klt, by \cite[Throrem 4.7]{KM98} $\pi$ and $\psi$ only contract $\mathbb{P}^1$s. So $F$ is RCC if and only if $F''$ is RCC. We define $\Delta''$ on $F''$ via the following
$$K_{F''}+\Delta''=\psi_*\pi^*(K_F+\Delta).$$
Then $(F,\Delta)$ being klt implies that $(F'',\Delta'')$ is klt.

We denote a general fiber of $q$ by $R$. By construction $R\cong \mathbb{P}^1$, so we know that $(R,\Delta''|_R)$ is klt and hence strongly $F$-regular. Then by applying Lemma \ref{basecurverational} on $F''$ we know that $B=\mathbb{P}^1$. So $F$ is rational. Therefore we have proven that the general fibers of $f$ are RCC.

Since we have assumed that the base field $k$ is uncountable, by \cite[Ch. IV Corollary 3.5.2]{Kollar96} we know that every fiber of $f$ is RCC.  
\end{proof}
\section{A canonical bundle formula for threefolds in positive characteristics}\label{CBFsection}
In this section following the idea of the proof of \cite{PS09} we construct a canonical bundle formula in characteristic $p$ for a morphism from a threefold to a surface, whose general fibers are $\mathbb{P}^1$. There are similar constructions in \cite[6.7]{CTX13} and \cite[Theorem 4.8]{DH15}. 

Let $\overline{\mathcal{M}}_{0,n}$ be the moduli space of $n$-pointed stable curves of genus $0$, $f_{0,n}:\overline{\mathcal{U}}_{0,n}\to\overline{\mathcal{M}}_{0,n}$ the universal family, and $\mathcal{P}_1,\mathcal{P}_2,...,\mathcal{P}_n$ the sections of $f_{0,n}$ which correspond to the marked points. Let $d_j (j=1,2,...,n)$ be the rational numbers such that $0<d_j\le 1$ for all $j$, $\sum_jd_j=2$ and $\mathcal{D}=\sum_jd_j\mathcal{P}_j$.
\begin{lemma}\cite[Lemma 4.6]{DH15}\cite[Theorem 2]{Kawamata97}\label{factuniversalfamilymoduli}
\begin{enumerate}
\item There exists a smooth projective variety $\mathcal{U}_{0,n}^*$, a $\mathbb{P}^1$-bundle $g_{0,n}:\mathcal{U}_{0,n}^*\to\overline{\mathcal{M}}_{0,n}$, and a sequence of blowups with smooth centers
$$\overline{\mathcal{U}}_{0,n}=\mathcal{U}^{(1)}\xrightarrow{\sigma_2}\mathcal{U}^{(2)}\xrightarrow{\sigma_3}...\xrightarrow{\sigma_{n-2}}\mathcal{U}^{(n-2)}=\mathcal{U}_{0,n}^*$$
\item Let $\sigma:\overline{\mathcal{U}_{0,n}}\to\mathcal{U}_{0,n}^*$ be the induced morphism, and $\mathcal{D}^*=\sigma_*\mathcal{D}$. Then $K_{\overline{\mathcal{U}}_{0,n}}+\mathcal{D}-\sigma^*(K_{\mathcal{U}_{0,n}^*}+\mathcal{D}^*)$ is effective.
\item There exists a semi-ample $\mathbb{Q}$-divisor $\mathcal{L}$ on $\overline{\mathcal{M}}_{0,n}$ such that 
$$K_{\mathcal{U}_{0,n}^*}+\mathcal{D}^*\sim_{\mathbb{Q}}g_{0,n}^*(K_{\overline{\mathcal{M}}_{0,n}}+\mathcal{L}).$$
\end{enumerate}
\end{lemma}
\begin{definition}\label{defdivmod}
Let $f:X\to Y$ be a surjective proper morphism between two normal varieties and $K_X+D\sim_{\mathbb{Q}}f^*L$, where $D$ is a boundary divisor on $X$ and $L$ is a $\mathbb{Q}$-Cartier $\mathbb{Q}$-divisor on $Y$. Let $(X,D)$ be log canonical near the generic fiber of $f$, i.e., $(f^{-1}U, D|_{f^{-1}U})$ is log canonical for some Zariski dense open subset $U\subseteq Y$. We define 
$$D_{\rm div}:=\sum(1-c_Q)Q,$$
where $Q\subset Z$ are prime Weil divisors on $Z$ and
$$c_Q={\rm sup}\{c\in\mathbb{R}:(X,D+cf^*Q) {\rm\ is\ log\ canonical\ over\ the\ generic\ point\ }\eta_Q {\rm\ of\ } Q\}.$$
Next we define
$$D_{\rm mod}:=L-K_Y-D_{\rm div},$$
so $K_X+D=f^*(K_Y+D_{\rm div}+D_{\rm mod}).$
\end{definition}
\begin{theorem} \label{CanBund}
Let $f:X\to Y$ be a proper surjective morphism, where $X$ is a normal threefold and $Y$ is a normal surface over an algebraically closed field $k$ of characteristic $p>0$. Assume that $Q=\sum_i Q_i$ is a divisor on $Y$ such that $f$ is smooth over $(Y-{\rm Supp}(Q))$ with fibers isomorphic to $\mathbb{P}^1$. Let $D=\sum_i d_iD_i$ be a $\mathbb{Q}$-divisor on $X$ where $d_i=0$ is allowed, which satisfies the following conditions: 
\begin{enumerate}
\item $(X,D\ge 0)$ is klt on a general fiber of $f$. 
\item Suppose $D=D^h+D^v$ where $D^h$ is the horizontal part and $D^v$ is the vertical part of $D$. Then $p={\rm char}(k)>\dfrac{2}{\delta}$, where $\delta$ is the minimum non-zero coefficient of $D^h$.
\item $K_X+D\sim_{\mathbb{Q}}f^*(K_Y+M)$ for some $\mathbb{Q}$-Cartier divisor $M$ on $Y$.
\end{enumerate}
Then we have that $D_{\rm mod}$ is $\mathbb{Q}$-linearly equivalent to an effective $\mathbb{Q}$-divisor. Here $D_{\rm mod}$ is defined as in Definition \ref{defdivmod}. Moreover if $(X,D)$ is klt then there exists an effective $\mathbb{Q}$-divisor $\overline{D}_{\rm mod}$ on $Y$ such that $\overline{D}_{\rm mod}\sim_{\mathbb{Q}}D_{\rm mod}$ and $(Y,D_{\rm div}+\overline{D}_{\rm mod})$ is klt. 
\end{theorem}
\begin{proof}
First we reduce the problem to the case where all components of $D^h$ are sections. Let $D_{i_0}$ be a horizontal component of $D$ and $D_{i_0}\to D_{i_0}^{\flat}\to Y$ be the Stein factorization of $f|_{D_{i_0}}$. Let $Y'\to D_{i_0}^{\flat}$ be the normalization of $D_{i_0}^{\flat}$, then $Y'\to Y$ is a finite surjective morphism of normal surfaces. Let $X'$ be the normalization of the component of $X\times_Y Y'$ dominating $Y$.
\begin{center}
\begin{tikzpicture}[scale=1.6]
\node (A) at (0,0) {$Y$};
\node (B) at (1,0) {$Y'$};
\node (C) at (0,1) {$X$};
\node (D) at (1,1) {$X'$};  
\path[->,font=\scriptsize]
(B) edge node[above]{$\nu$} (A)
(D) edge node[left]{$f'$} (B)
(C) edge node[left]{$f$} (A)
(D) edge node[above]{$\nu'$} (C);
\end{tikzpicture} 
\end{center}  
Let $m={\rm deg(\mu:Y'\to Y)}$ and $l$ be a general fiber of $f$. Then
\begin{align}\label{2}
m=D_i\cdot l\le \dfrac{1}{d_i}(D\cdot l)=\dfrac{1}{d_i}(-K_X\cdot l)=\dfrac{2}{d_i}\le\dfrac{2}{\delta}<{\rm char}(k).
\end{align}  
Therefore $\nu$ is a separable and tamely ramified morphism.

Let $D'$ be the log pullback of $D$ under $\nu'$, i.e.
$$K_{X'}+D'=\nu'^*(K_X+D).$$
More precisely by \cite[20.2]{Kollar92} we have
$$D'=\sum_{i,j}d_{ij}'D_{ij}',\ \ \ \nu'(D_{ij}')=D_{i}.\ \ \ d_{ij}'=1-(1-d_i)e_{ij},$$
where $e_{ij}$ is the ramification indices along $D_{ij}'$.

By construction $X$ dominates $Y$. Also, since $\nu$ is \'{e}tale over a dense open subset of $Y$, say $\nu^{-1}U\to U$, and \'{e}tale morphisms are stable under base change, $(f'\circ\nu)^{-1}U\to f^{-1}U$ is \'{e}tale. Thus the ramification locus $\Lambda$ of $\nu'$ does not contain any horizontal divisor $f'$, i.e. $f'(\Lambda)\ne Y'$. Therefore $D'$ is a boundary near the generic fiber of $f'$, i.e. $D'^h$ is effective. We observe that the coefficients of $D'^h$ can be computed by intersecting with a general fiber of $f':X'\to Y'$, hence they are equal to the coefficient of $D^h\subseteq X$. Thus the condition $p>\dfrac{2}{\delta}$ remains true for $D'$ on $X'$.

After finitely many such base changes we get a family $f'':X''\to Y''$, such that all of the horizontal components of $D''$ are rational sections of $f''$. Here $D''$ is the log pullback of $D$ via the induced finite morphism $\alpha:X''\to X$, i.e. $K_{X''}+D''=\alpha^*(K_X+D)$. 

By construction of $\overline{\mathcal{M}}_{0,n}$ there is a generically finite rational map $Y''\dasharrow\overline{\mathcal{M}}_{0,n}$. Let $\beta_0:\tilde{Y}\to Y''$ be a morphism that resolves the indeterminacies of $Y''\to\overline{\mathcal{M}}_{0,n}$ and $\tilde{X}$ the normalization of $X''\times_{Y''}\tilde{Y}$. We have a morphism $\tilde{Y}\to\overline{\mathcal{M}}_{0,n}$ and let $\hat{X}=\tilde{Y}\times_{\overline{\mathcal{M}}_{0,n}}\overline{\mathcal{U}}_{0,n}$. Let $X^{\sharp}$ be a common resolution of $\tilde{X}$ and $\hat{X}$. We have the following diagram:
\begin{center}
\begin{tikzpicture}[scale=1.6]
\node (A) at (-1,0) {$Y$};
\node (B) at (2,0) {$\tilde{Y}$};
\node (C) at (4,0) {$\overline{\mathcal{M}}_{0,n}$};
\node (D) at (-1,1) {$X$};  
\node (E) at (1,1) {$\tilde{X}$};  
\node (F) at (3,1) {$\hat{X}$};  
\node (G) at (4,1) {$\overline{\mathcal{U}}_{0,n}$};  
\node (H) at (5,1) {$\mathcal{U}_{0,n}^*$};  
\node (I) at (2,2) {$X^{\sharp}$};  
\node (J) at (0,0) {$Y''$};  
\node (K) at (0,1) {$X''$};  
\path[->,font=\scriptsize]
(B) edge node[above]{$\phi_0$} (C)
(D) edge node[left]{$f$} (A)
(E) edge node[pos=0.5,xshift=-10pt]{$\tilde{f}$} (B)
(F) edge node[pos=0.5,xshift=8pt]{$\hat{f}$} (B)
(G) edge node[left]{$f_{0,n}$} (C)
(H) edge node[right]{$g_{0,n}$} (C)
(G) edge node[above]{$\sigma$} (H)
(I) edge node[above]{$\lambda$} (E)
(I) edge node[above]{$\mu$} (F)
(E) edge node[below]{$\beta$} (K)
(K) edge node[below]{$\alpha$} (D)
(B) edge node[below]{$\beta_0$} (J)
(J) edge node[below]{$\alpha_0$} (A)
(K) edge node[pos=0.5,xshift=-6pt]{$f''$} (J)
(F) edge node[above]{} (G)
(I) edge node[xshift=-5pt, yshift=10pt]{$f^{\sharp}$} (B);
\path[dashed,->] (E) edge (F);
\draw[->] (I) edge[out=180,in=60] node[pos=0.55,yshift=5pt] {$\pi$} (D);
\draw[->] (F) edge[out=40,in=140] node[pos=0.5,yshift=8pt] {$\hat{\phi}$} (H);
\draw[->] (E) edge[out=150,in=30] node[pos=0.5,yshift=8pt] {$\psi$} (D);
\draw[->] (B) edge[out=210,in=-40] node[pos=0.5,yshift=-8pt] {$\psi_0$} (A);
\end{tikzpicture} 
\end{center}  
Let $D^{\sharp}$ and $\hat{D}$ be $\mathbb{Q}$-divisors on $X^{\sharp}$ and $\hat{X}$ respectively, defined by
$$K_{X^{\sharp}}+D^{\sharp}=\pi^*(K_X+D)$$
and
$$K_{\hat{X}}+\hat{D}=\mu_*(K_{X^{\sharp}}+D^{\sharp}).$$
We also define $D_{\rm mod}''$ and $D_{\rm div}''$ on $Y''$ for $(X'', D'')$ as in Definition \ref{defdivmod}, such that 
$$K_{X''}+D''=f''^*(K_{Y''}+D_{\rm mod}''+D_{\rm div}''),$$
and we define $\tilde{D}_{\rm mod}$ and $\tilde{D}_{\rm div}$ on $\tilde{Y}$ in a similar way.
Since $K_{X^{\sharp}}+D^{\sharp}$ is the pullback of some $\mathbb{Q}$-divisor from the base $\tilde{Y}$ we get
$$K_{X^{\sharp}}+D^{\sharp}=\mu^*(K_{\hat{X}}+\hat{D}).$$
Since $D_{\rm div}$ does not depend on the birational modification of the family (see \cite[Remark 7.3]{PS09}), we will define it with respect to $\hat{f}:\hat{X}\to \tilde{Y}$. 

Since $\hat{\phi}$ is generically finite and $\mathcal{D}^*$ is horizontal it follows that $\hat{\phi}^*\mathcal{D}^*$ is horizontal too. Since $\hat{D}^h$ is also horizontal one sees that
\begin{align}\label{Dh}
\hat{D}^h=\hat{\phi}^*\mathcal{D}^*.
\end{align}
From the construction of $\sigma:\overline{\mathcal{U}}_{0,n}\to\mathcal{U}_{0,n}^*$ we see that $(F,\mathcal{D}^*|_F)$ is log canonical for any fiber $F$ of $g_{0,n}:\mathcal{U}_{0,n}^*\to\overline{\mathcal{M}}_{0,n}$. Since the fibers of $\hat{f}:\hat{X}\to\tilde{Y}$ are isomorphic to the fiber of $g_{0,n}$, $(\hat{F},\hat{D}^h|_{\hat{F}})$ is also log canonical, where $\hat{F}$ is any fiber of $\hat{f}$. Let $\hat{D}_i^v$ be a component of $\hat{D}^v$ and $\eta$ the generic point of $\hat{f}(\hat{D}_i^v)$.
Then by inversion of adjunction we know that $(\hat{X}_{\eta},(\hat{D}_i^v+\hat{D}^h)|_{\eta})$ is log canonical. Since the fibers of $\hat{f}$ are reduced, the log canonical threshold of $(\hat{X},\hat{D}; \hat{D}_i^v)$ over the generic point of $\hat{D}_i^v$ is $(1-{\rm coeffi}_{\hat{D}_i^v}\hat{D})$. Hence we get 
$\hat{D}^v=\hat{f}^*\tilde{D}_{\rm div}$. Note that the coefficients of $\hat{D}^v$ can be $>1$.
By definition of $\tilde{D}_{\rm mod}$ we have
\begin{align}\label{Kx+Dh}
K_{\hat{X}}+\hat{D}^h\sim_{\mathbb{Q}}\hat{f}^*(K_{\tilde{Y}}+\tilde{D}_{\rm mod}).
\end{align}
Then we have
\begin{align}\label{equi0}
K_{\hat{X}}+\hat{D}^h-f^*(K_{\tilde{Y}}+\phi_0^*\mathcal{L})=K_{\hat{X}/\tilde{Y}}+\hat{D}^h-\hat{\phi}^*K_{\mathcal{U}_{0,n}^*/\overline{\mathcal{M}}_{0,n}}-\hat{\phi}^*\mathcal{D}^*\sim_{\mathbb{Q}}0,
\end{align}
where the first equality follows from \eqref{Kx+Dh} and Lemma \ref{factuniversalfamilymoduli} (3), and the second relation from \eqref{Dh} and \cite[Chapter 6 Theorem 4.9 (b) and Example 3.18]{Liu02}.

Since $\hat{f}$ has connected fibers, by \eqref{Kx+Dh} and \eqref{equi0} and projection formula we get
\begin{align}\label{Dmod}
\tilde{D}_{\rm mod}\sim_{\mathbb{Q}}\phi_0^*\mathcal{L},
\end{align}
i.e. $\tilde{D}_{\rm mod}$ is semi-ample.

Now since $\alpha_0:Y''\to Y$ is a composition of finite morphisms of degree strictly less than ${\rm char}(k)$ and $\beta_0$ is a birational morphism, by \cite[Theorem 3.2 and Example 3.1]{Ambro99} we get
$$K_{Y''}+D_{\rm div}''\sim_{\mathbb{Q}}\alpha_0^*(K_Y+D_{\rm div})$$
and
$$K_{\tilde{Y}}+\tilde{D}_{\rm div}\sim_{\mathbb{Q}}\beta_0^*(K_{Y''}+D_{\rm div}'').$$
So $\alpha_0^*D_{\rm mod}\sim_{\mathbb{Q}}D_{\rm mod}'',$ and $\beta_0^*D_{\rm mod}''\sim_{\mathbb{Q}}\tilde{D}_{\rm mod}.$
By the projection formula we have
$$D_{\rm mod}''\sim_{\mathbb{Q}}\beta_{0,*}\tilde{D}_{\rm mod}.$$
Then since $\alpha_0$ is finite, 
$$\psi_{0,*}\tilde{D}_{\rm mod}\sim_{\mathbb{Q}}\alpha_{0,*}\beta_{0,*}\tilde{D}_{\rm mod}\sim_{\mathbb{Q}}\alpha_{0,*}D_{\rm mod}''\sim_{\mathbb{Q}}\alpha_{0,*}\alpha_0^*D_{\rm mod}\sim_{\mathbb{Q}}D_{\rm mod}.$$
Here we view the push-forward through $\alpha_0$ as push-forward of cycles. Therefore $D_{\rm mod}$ is $\mathbb{Q}$-linearly equivalent to an effective divisor.

Next we prove the second statement. Since $\alpha$ is finite, by \cite[Corollary 2.42]{Kollar13} we know that $(X'',D'')$ is klt, and as $\beta$, $\lambda$ and $\mu$ are birational we know that $(\hat{X},\hat{D})$ is sub-klt, in particular $\hat{D}^v$ has coefficients $<1$. Since $\hat{f}$ is a $\mathbb{P}^1$ fibration and $(\tilde{Y},\tilde{D}_{\rm div})$ is log smooth we have that $(\tilde{Y},\tilde{D}_{\rm div})$ is sub-klt. By construction $\tilde{D}_{\rm mod}$ is semi-ample, so by \cite[Theorem 1]{Tanaka15b} we know that $(\tilde{Y},\tilde{D}_{\rm div}+\tilde{D}_{\rm mod})$ is sub-klt up to $\mathbb{Q}$-linear equivalence. Then $K_{Y''}+D_{\rm mod}''+D_{\rm div}''\sim_{\mathbb{Q}}\beta_{0,*}(K_{\tilde{Y}}+\tilde{D}_{\rm div}+\tilde{D}_{\rm mod})$ is also sub-klt. Finally using \cite[Corollary 2.42]{Kollar13} again and the fact that $D_{\rm mod}+D_{\rm div}\ge 0$ we get that $(Y,D_{\rm mod}+D_{\rm div})$ is klt.
\end{proof}
\section{Global rational chain connectedness}
In this section we prove the following theorem.
\begin{theorem}\label{globalRCC}
Let $X$ be a projective threefold over an algebraically closed field $k$ of characteristic $p>0$, $f:X\to Y$ a projective surjective morphism from $X$ to a projective variety $Y$ such that $f_*\mathcal{O}_X=\mathcal{O}_Y$. Let $D$ be an effective $\mathbb{Q}$-divisor, and $X_{\overline{\eta}}$ the geometric generic fiber of $f$. Assume that the following conditions hold. 
\begin{enumerate}
\item $(X,D)$ is klt, $-K_X$ is big and $f$-ample, $K_X+D\sim_{\mathbb{Q}}0$ and the general fibers of $f$ are smooth.
\item $\displaystyle p> \frac{2}{\delta}$, where $\delta$ is the minmum non-zero coefficient of $D$.
\item $D=E+f^*L$ where $E$ is an effective $\mathbb{Q}$-Cartier divisor such that $p\nmid {\rm ind}(E)$, $(X_{\overline{\eta}},E|_{X_{\overline{\eta}}})$ is globally $F$-split, and $L$ is a big $\mathbb{Q}$-divisor on $Y$. \label{Fsplitcondition}
\item ${\rm dim}(Y)=1$ or $2$.
\end{enumerate}
Then $X$ is rationally chain connected.
\end{theorem}
\begin{remark}
The smoothness of the general fibers of $f$ holds in characteristic $p\ge 11$ when $\dim Y=1$ by \cite[Theorem 5.1 (2)]{Hirokado04}, and in characteristic $p\ge 5$ when $\dim Y=2$ by adjunction and a theorem of Tate (cf. \cite[Theorem 5.1]{Liedtke13}).
\end{remark}
\begin{proposition}\label{RCCupstair}
Let $f:X\to Y$ be a projective surjective morphism between normal varieties with $f_*\mathcal{O}_X=\mathcal{O}_Y$. Assume that the following conditions hold.
\begin{enumerate}
\item The general fibers of $f$ are isomorphic to $\mathbb{P}^1$.
\item $Y$ is rationally chain connected.
\end{enumerate}
Then $X$ is rationally chain connected.
\end{proposition}
\begin{proof}
The proof is essentially the same as \cite[Lemma 3.12 and Proposition 3.13]{GLPSTZ15}. We take two general points $x_1$, $x_2\in X$ and let $y_1=f(x_1)$, $y_2=f(x_2)$, by construction $f^{-1}(y_1)\cong f^{-1}(y_2)\cong\mathbb{P}^1$. By assumption $y_1$ and $y_2$ can be connected by a chain of rational curves, say $C_1,C_2,...,C_n$. Let $\overline{C_i}\to C_i$ be the normalization for each $C_i$, $S_i:=f^{-1}(C_i)$, $\overline{S_i}:=S_i\times_{\overline{C_i}}C_i$ and $g_i:\overline{S_i}\to S_i$ the induced morphisms. Now the morphism $\overline{S_i}\to\overline{C_i}$ is a flat projective morphism whose general fibers are $\mathbb{P}^1$, by \cite[Theorem]{dJS03} it has a section which we denote by $\tilde{C_i}$. Then $x_1$ and $x_2$ is connected by $f^{-1}(y_1)$, $f^{-1}(y_2)$, $g_i(\tilde{C_i})$ and the fibers of $f$ over the intersection points of $\{C_i\}$, which is a union of rational curves by \cite[Lemma 3.7]{Debarre01}.
\end{proof}
\begin{proof}[Proof of Theorem \ref{globalRCC}]
We first prove the following lemma.
\begin{lemma}\label{-KYample}
Under the condition of Theorem \ref{globalRCC}, $-K_Y$ is big.
\end{lemma}
\begin{proof} 
By assumption $m(K_{X_{\overline{\eta}}}+E|_{X_{\overline{\eta}}})\sim_{\mathbb{Q}} 0$ for sufficiently large and divisible $m$, in particular the $k(\overline{\eta})$-algebra $\bigoplus_{m\ge 0}H^0(am(K_{X_{\overline{\eta}}}+E|_{X_{\overline{\eta}}}))$ is finitely generated. On the other hand since $(X_{\overline{\eta}}, E|_{X_{\overline{\eta}}})$ is globally $F$-split we have that
$$S^0(X_{\overline{\eta}},\sigma(X_{\overline{\eta}}, E|_{X_{\overline{\eta}}})\otimes\mathcal{O}_{X_{\overline{\eta}}}(m(K_{X_{\overline{\eta}}}+E|_{X_{\overline{\eta}}})))=H^0(X_{\overline{\eta}}, \mathcal{O}_{X_{\overline{\eta}}}(m(K_{X_{\overline{\eta}}}+E|_{X_{\overline{\eta}}}))).$$ 
Here we would like to mention that for a line bundle $M$ and a $\mathbb{Q}$-Cartier divisor $\Delta$, the notation $S^0(X,\Delta,M)$ is the same as the standard notation $S^0(X,\sigma(X,\Delta)\otimes M)$ (cf. \cite[between Lemma 2.2 and Proposition 2.3]{HX15}. Therefore by \cite[Theorem 1.1]{Ejiri15} we know that 
$$f_*\mathcal{O}_X(am(K_{X/Y}+E))\cong f_*\mathcal{O}_X(f^*(-am(K_Y+L)))=\mathcal{O}_Y(-am(K_Y-L))$$
is weakly positive for $m\gg 0$. By Lemma \ref{WPimpliesnef}, $-K_Y-L$ is nef, so $-K_Y$ is big.
\end{proof}
Next we consider the following two cases. \\  \\
\emph{Case 1: $Y$ is 1-dimensional.} 

After possibly taking the normalization of $Y$ we can assume that $Y$ is smooth. Then Lemma \ref{-KYample} implies that $g(Y)=0$, i.e. $Y\cong\mathbb{P}^1$. 
Let $F$ be a general fiber of $f$. By assumption $F$ is smooth and $K_F$ is anti-ample, hence $F$ is separably rationally connected. By \cite[Theorem]{dJS03} we know that $f$ has a section which we denote by $s$. Then $s(Y)$ is a rational curve in $X$ which dominates $Y$. Therefore we get that $X$ is rationally chain connected. \\ \\
\emph{Case 2: $Y$ is 2-dimensional.}


By assumption, a general fiber of $f$ is isomorphic to $\mathbb{P}^1$. Now by Lemma \ref{-KYample} we know that $-K_Y$ is big. On the other hand since $(X,D)$ is klt, by Theorem \ref{CanBund} there is a nonzero effective $\mathbb{Q}$-Cartier divisor $M$ on $Y$ such that $K_Y+M\sim_{\mathbb{Q}} 0$ and $(Y,M)$ is klt. Then by the proof of \emph{Case 2} of Theorem \ref{relativeRCC} we know that $Y$ is rational. Finally by Proposition \ref{RCCupstair} we get that $X$ is rationally chain connected.
\end{proof}
\bibliographystyle{alpha}
\bibliography{P}  
\end{document}